\documentclass[a4paper,12pt]{article}

\usepackage{amsfonts}
\usepackage{amscd,color}
\usepackage{amsmath,amsfonts,amssymb,amscd}
\usepackage{indentfirst,graphicx,epsfig}
\usepackage{graphicx}
\usepackage{tikz}

\input{epsf}
\usepackage{epstopdf}
\usepackage{caption}

\newtheorem{theorem}{Theorem}[section]
\newtheorem{definition}[theorem]{Definition}

\newtheorem{claim}[theorem]{Claim}

\tikzstyle{snode}=[circle,draw=black,fill=white,thick, inner sep=0pt ,minimum size=1.2mm]
\tikzstyle{bnode}=[circle ,draw=black,fill=black,thick, inner sep=0pt ,minimum size=1.2mm]

\newenvironment {proof} {\noindent{\em Proof.}}{\hspace*{\fill}$\Box$\par\vspace{4mm}}
\newcommand{\ml}{l\kern-0.55mm\char39\kern-0.3mm}

\def\qed{\hfill \nopagebreak\rule{5pt}{8pt}}

\newenvironment{theorem-non}[1]{\trivlist \item [\hskip \labelsep {\bf #1}]\ignorespaces\it}{\endtrivlist}

\date{}
\title{Spanning trees of a claw-free graph whose reducible stems have few leaves}

\author{\small Pham Hoang Ha 
	\\[0.2cm]
	\small Department of Mathematics\\
	\small Hanoi National University of Education \\
	\small 136 XuanThuy Street, Hanoi, Vietnam \\
	{\small E-mail: ha.ph@hnue.edu.vn}	
}

\begin{document}
	\maketitle
	
	\begin{abstract}
		Let $T$ be a tree, a vertex of degree one is a \emph{leaf} of $T$
		and a vertex of degree at least three is a \emph{branch vertex} of
		$T$. For two distinct vertices $u,v$ of $T$, let $P_T[u,v]$ denote the unique path in $T$ connecting $u$ and $v.$ For a leaf $x$ of $T$, let $y_x$ denote the  nearest branch vertex to $x$. For
		every leaf $x$ of $T$, we remove the path $P_T [x, y_x)$ from $T$, where $P_T [x, y_x)$ denotes the
		path connecting $x$ to $y_x$ in $T$ but not containing $y_x$. The resulting subtree of $T$ is called the {\it reducible stem } of $T$. In this paper, we first use a new technique of Gould and Shull to state a new short proof for a result of Kano et al. on the spanning tree with a bounded number of leaves in a claw-free graph. After that, we use that proof to give a sharp sufficient condition for a claw-free graph having a spanning tree whose reducible stem has few leaves.
		
\noindent {\bf Keywords:} spanning tree, leaf, claw-free graph, reducible stem

\noindent {\bf AMS Subject Classification:} Primary 05C05, 05C70. Secondary 05C07, 05C69
	\end{abstract}
	
	\section{Introduction}
	In this paper, we only consider finite graphs without loops or multiple edges. Let $G$ be a
	graph with vertex set $V(G)$ and edge set $E(G)$. For any vertex
	$v\in V(G)$, we use $N_G(v)$ and $\deg_G(v)$ (or $N(v)$ and $\deg (v)$ if
	there is no ambiguity) to denote the set of neighbors of $v$ and the
	degree of $v$ in $G$, respectively. For any $X\subseteq V(G)$, we
	denote by $|X|$ the cardinality of $X$. Sometime, we denote by $|G|$ instead of $|V(G)|.$ We define
	$N_G(X)=\bigcup\limits_{x\in X}N_G(x)$ and $\deg_G(X)=\sum\limits_{x\in
		X}\deg_G(x)$. The subgraph of
	$G$ induced by $X$ is denoted by $G[X]$. We define $G-uv$ to be the
	graph obtained from $G$ by deleting the edge $uv\in E(G)$, and
	$G+uv$ to be the graph obtained from $G$ by adding a new edge $uv$ joining two non-adjacent vertices $u$ and $v$ of $G$.  For two vertices $u$ and $v$ of $G$, the distance between $u$ and $v$ in $G$ is denoted 
	by $d_{G}(u, v)$. We use $K_n$ to denote the
	complete graph on $n$ vertices. We write $A:=
	B$ to rename $B$ as $A$. We refer
	to~\cite{Di05} for terminology and notation not defined here.
	
	For an integer $m\geq 2,$ let $\alpha^{m}(G)$ denote the number defined by
	\begin{center} $
		\alpha^{m}(G)=\max\{ |S|:S\subseteq V(G),d_{G}(x,y)\geq m\, $ for all distinct vertices $ \,x,y\in S \}.$\end{center}
	For an integer $p\geq 2$, we define
	\begin{center}$\sigma_p^{m}(G)=\min\{\deg_G(S) : S\subseteq V(G), |S|=p, d_{G}(x,y)\geq m $  for all distinct vertices $ \;x,y\in S \}.$\end{center}
	For convenience, we define $\sigma^{m}_{p}(G)=+\infty$ if $\alpha^{m}(G)<p$. We note that,  $\alpha^{2}(G)$ is often written $\alpha(G)$, which is independence number of $G,$ and $\sigma_p^{2}(G)$ is often written $\sigma_{p}(G)$, which is minimum degree sum of $p$ independent vertices.  
	
	Let $T$ be a tree. A vertex of degree one is a \emph{leaf} of $T$
	and a vertex of degree at least three is a \emph{branch vertex} of
	$T$. The set of leaves of $T$
	is denoted by $L(T)$ and the set of branch vertices of $T$ is denoted by $B(T)$. The subtree $T-L(T)$ of $T$ is called the {\it stem} of $T$ and is denoted by $Stem(T)$. 
	
	For two distinct vertices $u,v$ of $T$, let $P_T[u,v]$ denote the unique path in $T$ connecting $u$ and $v.$ For a leaf $x$ of $T$, let $y_x$ denote the  nearest branch vertex to $x$. For
	every leaf $x$ of $T$, we remove the path $P_T [x, y_x)$ from $T$, where $P_T [x, y_x)$ denotes the
	path connecting $x$ to $y_x$ in $T$ but not containing $y_x$. Moreover, the path $P_T [x, y_x)$ is
	called the {\it leaf-branch path of $T$ incident to $x$ and $y_x$} and denoted by $B_x$.   The resulting subtree of $T$ is called the {\it reducible stem } of $T$ and denoted by $R\_Stem(T)$. 
	
	There are several sufficient conditions (such as the independence
	number conditions and the degree sum conditions) for  a
	graph $G$ to have  a spanning tree with a bounded number of leaves
	or branch vertices. Win~\cite{Wi} obtained the following theorem, which confirms a conjecture of Las Vergnas~\cite{LV71}. Beside that, recently, the author \cite{H} also gave an improvement of Win by giving an independence number condition for a graph having a spanning tree which covers a certain subset of $V(G)$ and has at most $l$ leaves.
	
	\begin{theorem}\label{t1}{\rm (Win~\cite{Wi})}
		Let $m\geq 1$ and $l\geq 2$ be integers and let $G$ be a $m$-connected graph.  If
		$\alpha(G)\leq m+l-1$, then $G$ has a spanning tree with at most $l$
		leaves.
	\end{theorem}
	Later, Broersma and Tunistra gave the following degree sum condition for a graph to have a spanning tree with at most $l$ leaves.
	\begin{theorem}\label{t2}{\rm (Broersma and Tuinstra~\cite{BT98})}
		Let $G$ be a connected graph and let $l\geq2$ be an integer. If
		$\sigma_2(G)\geq |G|-l+1$, then $G$ has a spanning tree with at most
		$l$ leaves.
	\end{theorem}
	Motivating by Theorem \ref{t1}, a natural question is whether we can find sharp sufficient conditions of $\sigma_{l+1}(G)$ for a connected graph $G$ has a few leaves. This question is still open. But, in certain graph classes, the answers have been determined.
	
	For a positive integer $t \geq 3,$ a graph $G$ is said to be  $K_{1,t}-$ free graph if it contains no $K_{1,t}$ as an induced subgraph. If $t=3,$ the $K_{1,3}-$ free graph is also called the claw-free graph.  About this graph class, Kano, Kyaw, Matsuda, Ozeki, Saito and Yamashita proved the following theorem.
	\begin{theorem}\label{thm1}{\rm (Kano et al.~\cite{KKMOSY12})}
		Let $G$ be a connected claw-free graph and let $l\geq2$ be an integer. If
		$\sigma_{l+1}(G)\geq |G|-l$, then $G$ has a spanning tree with at most
		$l$ leaves.
	\end{theorem}
	For other graph classes, we refer the readers to see \cite{CCH14}, \cite{CHH}, \cite{GS}, \cite{Ky09}, \cite{Ky11} and \cite{MMM} for examples.
	
	The first main purpose of this paper is to give a new short proof for Theorem \ref{thm1} base on the new technique of Gould and Shull in \cite{GS}.
	
	Moreover, many researchers studied spanning trees in connected graphs whose stems have a bounded number of leaves or branch vertices (see  \cite{HH}, \cite{KY}, \cite{KY15}, \cite{TZ} and \cite{Yan} for more details). We introduce here some results on spanning trees whose stems have a few leaves or branch vertices. 
	\begin{theorem}{\rm (Tsugaki and Zhang~\cite{TZ})}
		Let $G$ be a connected graph and let $k\geq 2$ be an integer. If $\sigma_{3}(G)\geq|G|-2k+1$, then $G$ have a spanning tree whose stem has at most $k$ leaves.
	\end{theorem}
	\begin{theorem}{\rm (Kano and Yan~\cite{KY})}\label{KY1}
		Let $G$ be a connected graph and let $k\geq 2$ be an integer. If either $\alpha^{4}(G)\leq k$ or $\sigma_{k+1}(G)\geq|G|-k-1$, then G has a spanning tree  whose stem has at most $k$ leaves.
	\end{theorem}
	\begin{theorem}{\rm (Yan \cite{Yan})} \label{thm-1}
		Let $G$ be a connected graph and $k \geq 0$ be an integer. If one of the following conditions holds, then $G$ have a spanning tree whose stem has at most $k$ branch vertices.
		\begin{enumerate}
			\item[{\rm (a)}]$\alpha^{4}(G)\leq k+2,$
			\item[{\rm (b)}]$\sigma^{4}_{k+3}(G)\geq |G|-2k-3.$
		\end{enumerate}
	\end{theorem}
	
	Recently, Ha, Hanh and Loan gave a sufficient condition for a graph to have a spanning tree whose reducible stem has few leaves. In particular, they proved the following theorem.
	\begin{theorem}{\rm (Ha et al. \cite{HHL})}\label{thm01}
		Let $G$ be a connected graph and let $k\geq 2$ be an integer.  If one of the following conditions holds, then $G$ has a spanning tree whose reducible stem has at most $k$ leaves. 
		\begin{enumerate}
			\item[{\rm (i)}] $\alpha(G) \leq 2k+2,$
			\item[{\rm (ii)}] $\sigma_{k+1}^4(G) \geq \left \lfloor \frac{\vert G \vert - k}{2}\right \rfloor.$
		\end{enumerate}
		Here, the notation $\lfloor r\rfloor$ stands for the biggest integer not exceed the real number $r.$
	\end{theorem}
	After that, Ha, Hanh, Loan and Pham also gave a sufficient condition for a graph to have a spanning tree whose reducible stem has few branch vertices.
	\begin{theorem}{\rm (Ha et al. \cite{HHLP})}\label{thm02}
		Let $G$ be a connected graph and let $k\geq 2$ be an integer. If the following conditions holds, then $G$ has a spanning tree $T$ whose reducible stem has at most $k$ branch vertices. 
		\begin{center} $	\sigma_{k+3}^4(G) \geq \bigg\lfloor \frac{\vert G \vert - 2k-2}{2}\bigg\rfloor.$ \end{center}
	\end{theorem}
	Very recently, Hanh stated the following theorem.
	\begin{theorem}{\rm (Hanh \cite{Hanh})}\label{thm03}
		Let $G$ be a connected claw-free graph and let $k \geq 2$ be an integer. If one of the following conditions holds, then $G$ has a spanning tree whose reducible stem has at most $k$ leaves. 
		\begin{enumerate}
			\item[{\rm (i)}] $\alpha(G) \leq 3k+2,$
			\item[{\rm (ii)}] $\sigma_{k+1}^6(G) \geq \left \lfloor \frac{\vert G \vert -4k-2}{2}\right \rfloor.$
		\end{enumerate}
	\end{theorem}
	The open question is whether we may give a sharp condition of $\sigma_{3k+3}(G)$ to show that $G$ has a spanning tree whose reducible stem has at most $k$ leaves.
	
	For the last purpose of this paper, we will give an affirmative answer to this question. In particular, we prove the following theorem.
	\begin{theorem}\label{thm2}
		Let $G$ be a connected claw-free graph and let $k$ be an integer ($k\geq 2$).  If  $\sigma_{3k+3}(G) \geq \left\vert G \right\vert-k$, then $G$ has a spanning tree whose reducible stem has at most $k$ leaves.
	\end{theorem}
\begin{figure}[h]
	\centering
	\includegraphics[width=0.7\textwidth]{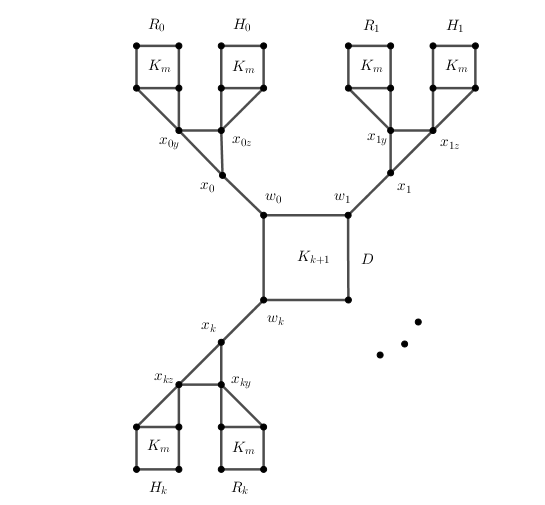}
	\caption[Graph G]{Graph G}
	\label{Pic1}
\end{figure}
	To show that our result is sharp, we will give the following example. Let $k \geq2$ and $m\geq 1$ be integers, and let $R_0,R_1,...,R_{k}$ and $H_0,H_1,...,H_{k}$ be $2k+2$ disjoint copies of the complete graph $K_m$ of order $m$. Let $D$ be a complete graph with $V(D)=\{w_i\}_{i=0}^{k}.$ Let $\{x_i, x_{iy}, x_{iz}\}_{i=0}^{k}$ be $3k+3$ vertices not contained in $\cup_{i=0}^k\bigg(V(R_i)\cup V(H_i)\cup \{w_i\}\bigg)$. Join $x_{iy}$ to all the vertices of $V(R_i)$ and $x_{iz}$ to all the vertices of $V(H_i)$ for every $0 \leq i \leq k$. Adding $3k+3$ edges $x_{i}x_{iy}, x_{i}x_{iz}, x_{iy}x_{iz}$ and joining $x_i$ to $w_i$ for every $0 \leq i \leq k$. Let $G$ denote the resulting graph (see Figure 1). Then, we have $|G|=(k+1)(2m+4).$ Moreover, take a vertex $u_i \in V(R_i)$ and a vertex $v_i\in V(H_i)$ for for every $0 \leq i \leq k.$ We obtain 
	\begin{eqnarray*}
		\sigma_{3k+3}(G)&=&\sum\limits_{i=0}^{k}\bigg(\deg_G(u_i)+\deg_G(v_i)+\deg_G(x_i)\bigg)\\  
		& =& (k+1)m +(k+1)m+3(k+1) = (k+1)(2m+3)\\
		&=& |G|-k-1.
	\end{eqnarray*}
	
	But $G$ has no spanning tree whose reducible stem has $k$ leaves. Hence the condition of Theorem \ref{thm2} is sharp.
	
	\section{A new proof of Theorem \ref{thm1}}
	Before begining to prove Theorem \ref{thm1} we recall some definitions in \cite{GS}.
	\begin{definition}[{\cite{GS}}]
		Let $T$ be a tree. For each $e\in E(T)$ and $u,v \in V(T)$, we denote $\{u_v\}=V(P_T[u,v])\cap N_T(u)$ and $e_v$ as the vertex incident to $e$ which is the nearest vertex of $v$ in $T$. 
	\end{definition}
	
	\begin{definition}[{\cite{GS}}]
		Let $T$ be a spanning tree of a graph $G$ and let $v\in V(G)$ and $e\in E(T)$. Denote $g(e,  v)$ as the vertex incident to $e$ farthest away from $v$ in $T$. We say $v$ is an {\it oblique neighbor} of $e$ { with respect to $T$} if $vg(e,  v)\in E(G)$.
	\end{definition}
	
	\begin{definition}[{\cite{GS}}]
		Let $T$ be a spanning tree of a graph $G$. Two vertices are {\it pseudoadjacent} with respect to $T$ if there is some $e\in E(T)$ which has them both as oblique neighbors. Similarly, a vertex set is {\it pseudoindependent} with respect to $T$ if no two vertices in the set are pseudoadjacent with respect to $T$. 
	\end{definition}
	We note here that pseudoadjacency (with respect to any tree) is a weaker condition than adjacency, while pseudoindependence is a stronger condition than independence.
	
	Now we are ready to prove Theorem \ref{thm1}.\\
	{\bf Proof of Theorem \ref{thm1}.} Suppose that $G$ has no spanning tree with at most $l$ leaves. Choose some spanning tree $T$ of $G$ such that:\\
	(T1) $|L(T)|$ is as small as possible.
	
	By the assumption, $T$ must have at least $l+1$ leaves.
	
	We have the following claims.
	
	\begin{claim}\label{claim2} 
		$L(T)$ is pseudoindependent with respect to  $T$. In particular, $L(T)$ is independent.
	\end{claim}
	\begin{proof}
		Suppose two leaves $s$ and $t$ are pseudoadjacent with respect to $T$. Then there is some edge $e\in E(T)$ such that  $sg(e,s), tg(e,t)\in E(G)$. Let $b$ and $u$ be the nearest branch vertices of  $s$ and $t$, respectively. Consider the following two cases.
		
		Case 1: Suppose $g(e,s)\not= g(e,t)$. Then $e_s=g(e,t)$ and $e_t=g(e,s)$, so $se_t, te_s\in E(G)$. We consider the spanning tree
		$$T':=\left\{\begin{array}{ll}T - e + se_t,
		& \;\mbox{ if } e=uu_t,\\
		T-\{e, uu_t\}+\{se_t, te_s\}, & \;\mbox{ if } e \not= uu_t.
		\end{array}\right.$$
		Hence, $|L(T')| < |L(T)|.$ This violates (T1). So case 1 does not happen.
		
		Case 2: Suppose $g(e,s)=g(e,t)$. Define $a:=g(e,s)=g(e,t)$. Then $e_s=e_t\not\in \{s, t\}$ and denoted by vertex $z$. Since $G[a, z, s, t]$ is not $K_{1,3}$-free , so we have either $st\in E(G)$ or $zt\in E(G)$ or $zs\in E(G)$. Consider the tree
		$$T':=\left\{\begin{array}{ll}T - uu_t + st,
		& \;\mbox{ if } st\in E(G),\\
		T-\{e, uu_t\}+\{zt, sa\}, & \;\mbox{ if } zt\in E(G),\\
		T-\{e, bb_s\}+\{zs, ta\}, & \;\mbox{ if } zs\in E(G).
		\end{array}\right.$$
		
		Hence, $|L(T')| < |L(T)|.$ This violates the condition (T1). So case 2 does not happen.\\
		Therefore,  the claim \ref{claim2} is proved.
	\end{proof}
	
	\begin{claim}\label{claim3}  
		For each branch vertex $ b \in B(T) $, there are at least $\deg_T(b)- 1$ edges of $T$ incident with $b$  such that they have no oblique neighbor in $ L(T) $.
	\end{claim}
	\begin{proof}
		Set $ N_{T}(b)=\{s_1, s_2, ..., s_q\}, q \geq 3.$\\
		Assume that there exist two vertices $s_i, s_j \in N_{T}(b)$ such that $s_is_t\not\in E(G) $ for all $ t\in \{1,...,q\}\setminus\{i\}$ and $s_js_t\not\in E(G) $ for all $ t\in \{1,...,q\}\setminus\{j\}.$ Then $G[b,s_i,s_j,s_t]$ is  $K_{1,3}$-free for every $t\in \{1,...,q\}\setminus \{i, j\}$. This is a contradiction. Therefore we conclude that there exists at most one vertex $s\in N_{T}(b)$ such that $ss_t\not\in E(G)$ for all $s_t\not= s.$ 
		
		Let $s_t\in N_{T}(b)\setminus \{s\}.$ Then there exists some vertex $s_i\in N_{T}(b)\setminus \{s_t, s\}$ such that $s_ts_i\in E(G).$ Set $e:=bs_t.$ To complete Claim \ref{claim3}, we will need only to prove that $x$ is not an oblique neighbor of $e$ with respect to $T$ for every $x\in L(T).$ Indeed, to the contrary, assume that there exists some vertex $x\in L(T)$ such that $x$ is an oblique neighbor of $e$ with respect to $T.$  Consider the tree
		$$T':=\left\{\begin{array}{ll}T - \{e,bs_i \} + \{bx,s_ts_i\},
		& \;\mbox{ if } g(e,x)=b, x \not= s_t,\\
		T - bs_i + s_ts_i,
		& \;\mbox{ if } g(e,x)=b, x =s_t,\\
		T - e + xs_t, & \;\mbox{ if } g(e,x)=s_t.
		\end{array}\right.$$
		Hence, $|L(T')| < |L(T)|.$ This violates with (T1). \\
		Claim \ref{claim3} holds.
	\end{proof}
	\begin{claim}\label{claim4} There are at least $l$ distinct edges of $T$ such that they have no oblique neighbor in $L(T).$
	\end{claim}
	\begin{proof}
		By Claim \ref{claim2}, we obtain that for each $e\in E(T),$ $e$ has at most an oblique neighbor in $L(T).$ Moreover, if an edge $e$ is incident with two branch vertices of $T$ then $e$ has to be an edge of the subgraph $T[\{b\}_{b\in B(T)}]$ of $T.$ Then, there are at most $|B(T)|-1$ edges which are adjacent with two branch vertices of $T.$ Hence, combining with Claim \ref{claim3}, there exist at least $$\sum\limits_{b\in B(T)}(\deg_T(b)-1)-|B(T)|+1=\sum\limits_{b\in B(T)}(\deg_T(b)-2)+1$$ distinct edges in $E(T)$ which have no oblique neighbor in $L(T)$.\\
		On the other hand, we have
		\begin{eqnarray*}
			|L(T)|&=&2+\sum\limits_{b\in B(T)}(\deg_T(b)-2)\\
			&\Rightarrow &\sum\limits_{b\in B(T)}(\deg_T(b)-2)+1 = |L(T)|-1\geq l.
		\end{eqnarray*}
		Therefore, the claim is proved.
	\end{proof}
	
	For any $v, x\in V(T)$, we now have $vx\in E(G)$  if and only if  $v$ is an oblique neighbor of $xx_v$ with respect $T$.  Therefore, the number of edges of $T$ with $v$ as an oblique neighbor equals the degree of $v$ in $G$. Combining with Claims \ref {claim2} and \ref {claim4}, we obtain that
	\[\sigma_{l+1}(G)\leq |E(T)|-l=|V(T)|-1-l=|G|-1-l,\]
	which contradicts the assumption of Theorem \ref {thm1}. The proof of Theorem \ref {thm1} is completed. \qed
	
	\section{Proof of theorem \ref{thm2}}

	For two distinct vertices $u,v$ of $T$, let $P_T[u,v]$ denote the unique path in $T$ connecting $u$ and $v.$ We define the \emph{orientation} of $P_T[u,v]$ is from $u$ to $v$. For each vertex $x \in V(P_T[u,v])$, we denote by $x^+$ and $x^-$ the successor and predecessor of $x$ in $P_T[u,v]$, respectively, if they exist. For any $X\subseteq V(G),$ set $(N(X)\cap P_T[u,v])^{-} = \{x^{-} \vert x \in V(P_T[u,v])\setminus \{u\}$ and $x \in N(X) \}$ and  $(N(X)\cap P_T[u,v])^{+} = \{x^{+} \vert x \in V(P_T[u,v])\setminus \{v\}$ and $ x \in N(X) \}.$  
	
	{\bf Proof of Theorem \ref{thm2}.} Suppose to the contrary that there does not exist a spanning tree $T$ of $G$ such that $\vert L(R\_Stem(T))\vert \leq k.$ Then every spanning tree $T$ of $G$ satisfies $\vert L(R\_Stem(T))\vert \geq k+1$.
	
	Choose $T$ to be a spanning tree of $G$ such that
	\begin{itemize}
		\item [$($C0$)$]
		$\vert L(R\_Stem (T)) \vert$ is as small as possible,
		\item [$($C1$)$]
		$\vert R\_Stem (T) \vert$ is as small as possible, subject to (C0),
		\item [$($C2$)$]
		$\vert L(T)\vert $ is as small as possible, subject to (C0) and (C1).
	\end{itemize}

	Set $L(R\_Stem(T))=\{x_1,x_2,...,x_{l}\}, l\geq k+1.$ By the definition of reducible stem of $T$, we have the following claim.
	\begin{claim}\label{claim2.2}
		For every $i\in \{1,2,...,l\}$, there exist at least two leaf-branch paths of $T$ which are incident to $x_i$.
	\end{claim}
	\begin{claim}\label{claim2.4}
		For each $i \in \{1,2,...,l\}$, there exist $y_i, z_i \in L(T)$ such that $B_{y_i}, B_{z_i}$ are incident to $x_i$ and $N_G(y_i) \cap (V(R\_Stem(T)) \setminus \{x_i\}) = \emptyset$ and $N_G(z_i) \cap (V(R\_Stem(T))\setminus \{x_i\}) = \emptyset$.
	\end{claim}
	
	\begin{proof} Let $\{a_{ij}\}_{j=1}^m$ be the  subset of $L(T)$ such that $B_{a_{ij}}$ is adjacent to $x_i$. By Claim \ref{claim2.2}, we obtain $m \geq 2.$ \\
		Suppose that there are more than $m-2$ vertices $\{a_{ij}\}_{j=1}^m$ satisfying
		\begin{center}
			$N_G(a_{ij}) \cap \left( V(R\_Stem(T))\setminus \{x_i\}\right) \neq \emptyset$.
		\end{center}
			Without lost of generality, we may assume that $N_G(a_{ij}) \cap \left( V(R\_Stem(T))\setminus \{x_i\}\right) \neq \emptyset$ for all $j=2,...,m.$ Set $b_{ij}\in N_G(a_{ij}) \cap \left(V(R\_Stem(T))\setminus \{x_i\}\right)$ and $v_{ij} \in N_T(x_i) \cap V(P_T[a_{ij}, x_i])$ for all $ j \in \{2,...,m\}$. Consider the spanning tree
		$$T':= T + \{a_{ij}b_{ij}\}_{j=2}^m-\{x_iv_{ij}\}_{j=2}^m.$$
		Then $T'$ satisfies $ |L(R\_Stem (T'))|\leq |L(R\_Stem (T))|$ and $\vert R\_Stem(T')\vert < \vert  R\_Stem(T)\vert,$ where $x_i$ is not in $V(R\_Stem(T')).$ This contradicts either the condition (C0) or the condition (C1). Therefore, Claim \ref{claim2.4} holds.
	\end{proof}
	\begin{claim}\label{claim2.6}
		For every $i, j \in \{1,2,...,k+1\}, i \neq j,$ if $u \in V(P_T[y_i,x_i])$ and $v \in V(P_T[y_j,x_j]), w\in V(P_T[z_j,x_j])$ then $uv\not\in E(G)$ and $uw\not\in E(G).$ In particular, we have $ N_G(y_i) \cap V(B_{y_j}) = \emptyset$ and $N_G(y_i) \cap V(B_{z_j}) = \emptyset$.
	\end{claim}
		\begin{figure}[h]
		\centering
		\includegraphics[width=0.7\textwidth]{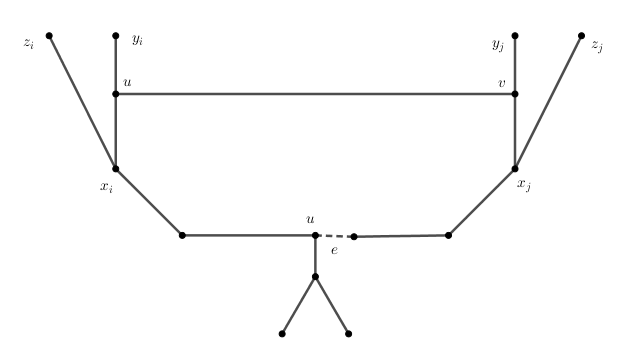}
		\caption[Tree T'']{Tree $T''$}
		\label{Pic2}
	\end{figure}
	\begin{proof}  By the same role of $y_j$ and $z_j,$ we only need to prove $ uv \not\in E(G).$ Suppose the assertion of the claim is false. Set $T' := T + uv.$ Then $T'$ is a subgraph of $G$ including a unique cycle $C,$ which contains both $x_i$ and $x_j$. 
		Since $k\geq 2$, then $ |L(R\_Stem (T))| \geq 3.$ Hence, we obtain $|B(R\_Stem(T))| \geq 1$. Then there exists a branch vertex of $R\_Stem(T)$ contained in $C.$ Let $e$ be an edge incident to such a vertex in $C$ and $R\_Stem(T)$. By removing the edge $e$ from $T'$ we obtain a spanning tree $T''$ (see Figure 2). Hence $T''$ satisfies $|L(R\_Stem(T''))|<|L(R\_Stem(T))|,$ the reason is that either $R\_Stem(T'')$ has only one new leaf and $x_i, x_j$ are not leaves of $R\_Stem(T'')$ or $x_i$ (or $x_j$) is still a leaf of $R\_Stem(T'')$ but $R\_Stem(T'')$ has no new leaf and $x_j$ (or $x_i$ respectively ) is not a leaf of $R\_Stem(T'')$. This is a contradiction with the condition (C0). So Claim \ref{claim2.6} is proved.
	\end{proof}
	We obtain the following claim as a corollary of Claim \ref{claim2.6}. 
	\begin{claim}\label{claim2.3}
		$L(R\_Stem(T))$ is an independent set in $G$.
	\end{claim}
	Set $U_1 =\{y_i,z_i\}_{i=1}^{l}$. For each $i \in \{1,...,l\}$ we also set $ x_{iy} \in N_T(x_i)\cap V(B_{y_i})$ and $x_{iz}\in N_T(x_i)\cap V(B_{z_i}).$ 
	\begin{claim}\label{claim2.5}
		$U_1$ is an independent set in $G$.
	\end{claim}
	\begin{proof}
		Suppose that there exist two vertices $u, v \in U_1$ such that $uv \in E(G).$ Without lost of generality, we may assume that $v=y_i$ for some $i\in \{1,2,...,l\}.$  Consider the spanning tree $T':= T +uy_i -x_{iy}x_i.$ Then  $|L(R\_Stem (T'))|\leq |L(R\_Stem (T))|.$ If $\deg_T(x_i)=3$ then $x_i$ is not a branch vertex of $T'.$ Hence $|R\_Stem(T')| < |R\_Stem(T)|,$ this contradicts either the conditions (C0) or (C1). Otherwise, we have $|L(R\_Stem (T'))|= |L(R\_Stem (T))|,$ $|R\_Stem(T')| = |R\_Stem(T)|$ and $|L(T')|< |L(T)|,$  where either $T'$ has only one new leaf and $y_i, u$ are not leaves of $T'$ or $y_i$ is still a leaf of $T'$ but $T'$ has no any new leaf and $u$ is not a leaf of $T'$. This contradicts the condition (C2). The proof of Claim \ref{claim2.5} is completed.
	\end{proof}
	Now, we choose $T$ to be a spanning tree of $G$ satisfying
	\begin{itemize}
		\item [$($C3$)$]
		$\sum_{i=1}^{l}\deg_T(x_i)$ is as small as possible, subject to (C0)-(C2),
		\item [$($C4$)$]
		$\sum_{i=1}^{l}\bigg( |B_{y_i}| + |B_{z_i}| \bigg)$ is as large as possible, subject to (C0)-(C3).
	\end{itemize}
	Set $U=U_1\cup L(R\_Stem(T)).$
	\begin{claim}\label{claim2.5.1}
		$U$ is an independent set in $G$.
	\end{claim}
	\begin{proof}
		Suppose that there exist two vertices $u, v \in U$ such that $uv \in E(G).$ By Claims \ref{claim2.3} and \ref{claim2.5}, without lost of generality, we may assume that $u\in L(R\_Stem(T))$ and $v=y_i\in U_1$ for some $i\in \{1,2,...,l\}.$ Moreover, by Claim \ref{claim2.6}, we now only need to consider the case $u=x_i.$\\ Set $t \in N_{T}(x_i)\cap V(R\_Stem(T)).$ 
		
		If $y_ix_{iz} \in E(G).$ Consider the spanning tree $T':= T+y_ix_{iz} - x_{iz}x_i.$ Then  $|L(R\_Stem (T'))|\leq |L(R\_Stem (T))|.$ If $\deg_T(x_i)=3$ then $x_i$ is not a branch vertex of $T'.$ Hence $|R\_Stem(T')| < |R\_Stem(T)|,$ this contradicts either the conditions (C0) or (C1). Otherwise, we have $|L(R\_Stem (T'))|= |L(R\_Stem (T))|,$ $|R\_Stem(T')| = |R\_Stem(T)|$ and $|L(T')|< |L(T)|,$ this contradicts the condition (C2). Now, since $G[x_i, t, x_{iz}, y_i]$ is not $K_{1,3}$-free, we obtain that $tx_{iz} \in E(G)$ or $ty_i\in E(G).$ We consider the spanning tree 
		$$T':=\left\{\begin{array}{ll}T+tx_{iz} - x_{iz}x_i,
		& \;\mbox{ if } tx_{iz} \in E(G),\\
		T+ty_i - x_{iy}x_i, & \;\mbox{ if } ty_i\in E(G).
		\end{array}\right.$$
			If $\deg_T(x_i)=3$ then we obtain $|L(R\_Stem(T'))|\leq |L(R\_Stem(T))|$ and $|R\_Stem(T')| < |R\_Stem(T)|,$ a contradiction with (C0) or (C1). Otherwise, we have $L(R\_Stem(T'))= L(R\_Stem(T))=\{x_i\}_{i=1}^{k+1},$ $|R\_Stem(T')| = |R\_Stem(T)|,$ $|L(T')|=|L(T)|$ and $\sum_{i=1}^{l}\deg_{T'}(x_i)<\sum_{i=1}^{l}\deg_T(x_i).$ This also violates the condition (C3).\\	
		Therefore, the proof of Claim \ref{claim2.5.1} is completed.
	\end{proof}
	\vskip 0.2cm
	By Claim \ref{claim2.5.1} we conclude that  $\alpha(G)\geq 3l \geq 3k+3.$ 
	\begin{claim}\label{claim2.7}
		For every $ p \in L(T)\setminus U_1,$ then $ \sum_{u\in U}|N_G(u) \cap V(B_p) | \leq \vert B_p\vert$.
	\end{claim}
		\begin{figure}[h]
		\centering
		\includegraphics[width=0.7\textwidth]{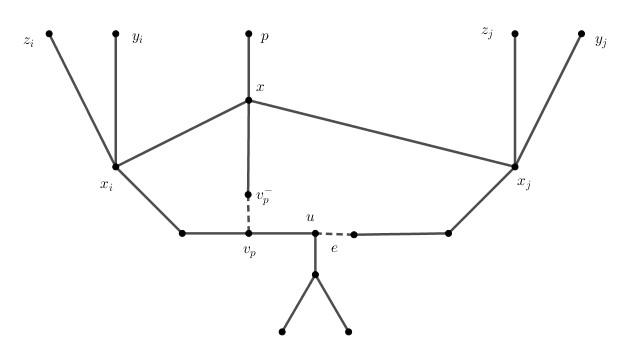}
		\caption[Tree T']{Tree $T'$}
		\label{Pic3}
	\end{figure}
	\begin{proof}  Set $v_p \in B(T)$ such that $\left( V(P_T[p,v_p]) \setminus \{v_p\}\right)\cap B(T) = \emptyset$. Let $V(B_p)\cap N_T(v_p)=\{v_p^{-}\}.$ Then we consider $B_p=P_T[p,v_p^{-}].$ 
		
		Assume that there exists a vertex $x \in V(B_p)$ such that $xu \in E(G)$ for some $u \in U_1.$ Consider the spanning tree
		$$T':=\left\{\begin{array}{ll}T+xu-v_p^{-}v_p,
		& \;\mbox{ if } x\in \{v_p^{-}, p\},\\
		T+xu -xx^+, & \;\mbox{ if } x \not\in \{v_p^{-},p\}.
		\end{array}\right.$$
		This contradicts  either the condition (C2) if $x\in \{v_p^{-}, p\}$  or the condition (C4) for otherwise. Therefore, we conclude that $
		\sum_{u\in U_1}|N_G(u) \cap V(B_p) |=0.
		$
		
		Assume that there exist $x_i, x_j\in L(R\_Stem(T))$ for some $i\not= j$ and $x\in V(B_p)$ such that $xx_i, xx_j \in E(G).$ Set 
		$$G':=\left\{\begin{array}{ll}T+xx_j,
		& \;\mbox{ if } x_i=v_p,\\
		T +\{xx_i, xx_j\}-\{v_pv_p^{-}\}, & \;\mbox{ if } x_i \not= v_p.
		\end{array}\right.$$
		Then $G'$ is a subgraph of $G$ including a unique cycle $C,$ which contains both $x_i$ and $x_j$. Since $k\geq 2$, then $ |L(R\_Stem (T))| \geq 3.$ Hence, we obtain $|B(R\_Stem(T))| \geq 1$. Then there exists a branch vertex of $R\_Stem(T)$ contained in $C.$ Let $e$ be an edge incident to such a vertex in $C$. By removing the edge $e$ from $G'$ we obtain a spanning tree $T'$ of $G$ satisfying $|L(R\_Stem(T'))|<|L(R\_Stem(T))|,$ the reason is that either $R\_Stem(T')$ has only one new leaf and $x_i, x_j$ are not leaves of $R\_Stem(T')$ or $x_i$ (or $x_j$) is still a leaf of $R\_Stem(T')$ but $R\_Stem(T')$ has no new leaf and $x_j$ (or $x_i$ respectively) is not a leaf of $R\_Stem(T')$ (see Figure 3 for an example). This is a contradiction with the condition (C0). Therefore, we concludes that $
		\sum_{u\in L(R\_Stem(T))}|N_G(u) \cap \{x\} |\leq 1
		$ for every $x \in V(B_p).$
		Now we obtain the following
		$$
		\sum_{u\in U}|N_G(u) \cap V(B_p) | = \sum_{u\in U_1}|N_G(u) \cap V(B_p) |+\sum_{u\in L(R\_Stem(T))}|N_G(u) \cap V(B_p) |\leq |B_{p}|.
		$$
		Claim \ref{claim2.7} is proved.
	\end{proof}
	
	\begin{claim}\label{claim2.8}
		For every $1 \leq i \leq k+1$, then $
		\sum_{u\in U}|N_G(u) \cap V(B_{y_i}) | \leq  \vert B_{y_i} \vert $ and $\sum_{u\in U}|N_G(u) \cap V(B_{z_i}) | \leq  \vert B_{z_i}\vert .$
	\end{claim}
	
	\begin{proof}
		By the same role of $y_i$ and $z_i,$ we only need to prove $$
		\sum_{u\in U}|N_G(u) \cap V(B_{y_i}) | \leq  \vert B_{y_i} \vert.$$ We consider $B_{y_i}=P_T[y_i,x_{iy}].$ 
		
		By Claim~\ref{claim2.6}, we obtain the following.\\
		\noindent{\it Subclaim \ref{claim2.8}.1.} $ N_G(U_1) \cap V(B_{y_i}) =  N_G(\{y_i,z_i\}) \cap V(B_{y_i})$. 
		
		\noindent{\it Subclaim \ref{claim2.8}.2.} We have $x_{iy}x_{iz}\in E(G).$
		
		Indeed, if $x_{iy}x_{iz}\not\in E(G)$ we set $t\in N_{T}(x_i)\cap V(R\_Stem(T)).$ Then since $G[x_i, t, x_{iy}, x_{iz}]$ is not $K_{1,3}$-free we obtain either $x_{iy}t\in E(G)$ or $x_{iz}t\in E(G)$. Without loss of generality, we may assume that $x_{iy}t\in E(G).$ Consider the spanning tree  $T'=T-x_{i}x_{iy}+x_{iy}t.$ If $\deg_T(x_i)=3$ then we obtain $|L(R\_Stem(T'))|\leq |L(R\_Stem(T))|$ and $|R\_Stem(T')| < |R\_Stem(T)|,$ a contradiction with (C0) or (C1). Otherwise, we have $L(R\_Stem(T'))= L(R\_Stem(T)),$ $|R\_Stem(T')| = |R\_Stem(T)|,$ $|L(T')|=|L(T)|$ and $\sum_{i=1}^{l}\deg_{T'}(x_i)<\sum_{i=1}^{l}\deg_T(x_i).$ This violates the conditions (C3).  Subclaim \ref{claim2.8}.2 is proved.
		
		\noindent{\it Subclaim \ref{claim2.8}.3.} If $x \in N_G(y_i)\cap V(B_{y_i})$ then $x^- \notin N_G(z_i)\cap V(B_{y_i})$.
		
		Suppose that there exists $x \in N_G(y_i)\cap B_{y_i}$ such that $ x^- \in N_G(z_i)\cap B_{y_i}$. Consider the spanning tree $T' := T+\{xy_i,z_ix^-\}-\{xx^-,x_{iy}x_i\}$. Then $|L(R\_Stem(T'))|\leq |L(R\_Stem(T))|.$ If $\deg_T(x_i)=3$ then $x_i$ is not a branch vertex of $T'.$ Hence $|R\_Stem(T')| < |R\_Stem(T)|,$ this contradicts the condition (C1). Otherwise, we have $|L(R\_Stem(T'))|= |L(R\_Stem(T))|,$ $|R\_Stem(T')| = |R\_Stem(T)|$ and $|L(T')|< |L(T)|.$ This is a contradiction with the condition (C2). Therefore,
		Subclaim \ref{claim2.8}.3 holds.
		
		\noindent{\it Subclaim \ref{claim2.8}.4.} If $x \in N_G(y_i)\cap V(B_{y_i})$ then $x^- \notin N_G(x_i)\cap V(B_{y_i})$.
		
		Suppose that there exists $x \in N_G(y_i)\cap V(B_{y_i})$ such that $ x^- \in N_G(x_i)\cap V(B_{y_i})$ for some $w\in L(R\_Stem(T))$. By Subclaim \ref{claim2.8}.2, consider the spanning tree $T' := T+\{xy_i,x_ix^-,x_{iy}x_{iz}\}-\{xx^-,x_{i}x_{iy}, x_ix_{iz}\}$. Then  $|L(R\_Stem(T'))|\leq |L(R\_Stem(T))|$ and $|R\_Stem(T')| \leq |R\_Stem(T)|,$ this contradicts the conditions either (C0) or (C1). Otherwise, we have $|L(R\_Stem(T'))|= |L(R\_Stem(T))|,$ $|R\_Stem(T')| = |R\_Stem(T)|$ and $|L(T')|< |L(T)|.$ This contradicts with the condition (C2). Therefore,
		Subclaim \ref{claim2.8}.4 holds.\\
		
		\noindent{\it Subclaim \ref{claim2.8}.5.} We have $x_{iy} \notin N_G(z_i).$
		
		Indeed, suppose to the contrary that $x_{iy}z_i\in E(G).$ We consider the spanning tree $T':= T+x_{iy}z_i-x_ix_{iy}.$ Hence, $T'$ is a spanning tree of $G$ satisfying $|L(R\_Stem(T'))|\leq|L(R\_Stem(T))|,$ $ |R\_Stem(T')|\leq|R\_Stem(T)|$ and $|L(T')|<|L(T)|,$ where $z_i$ is not a leaf of $T'.$ This contradicts the conditions (C0), (C1) or (C2). Subclaim \ref{claim2.8}.5 is proved.
		
		\noindent{\it Subclaim \ref{claim2.8}.6.} If $x \in N_G(z_i)\cap V(B_{y_i})$ then $x_ix \notin E(G)$.
		
		Indeed, assume that $x_ix\in E(G).$ By Subclaim \ref{claim2.8}.5 and Claim \ref{claim2.5.1}, we obtain $x_iz_i \not\in E(G)$ and there exists $x^{+}.$ Combining with $G[x, x_i,x^{+}, z_i]$ is not $K_{1,3}$-free we get $x^{+}x_i\in E(G)$ or $x^{+}z_i \in E(G).$\\
		If $ x^+x_i \in E(G).$ Combining with Subclaim \ref{claim2.8}.2, we consider the spanning tree
		$$T':=\left\{\begin{array}{ll}T+\{x_{iy}x_{iz},z_ix\}-\{xx^+, x_ix_{iz}\}
		& \;\mbox{ if } x^+=x_{iy},\\
		T+\{x^{+}x_i,x_{iy}x_{iz},z_ix\}-\{xx^+,x_ix_{iy}, x_ix_{iz}\}
		& \;\mbox{ if } x^+\not=x_{iy}.\\
		\end{array}\right.$$
		Hence  $|L(R\_Stem(T'))|\leq |L(R\_Stem(T))|.$ If $\deg_T(x_i)=3$ then $x_i$ is not a branch vertex of $T'.$ Hence $|R\_Stem(T')| < |R\_Stem(T)|,$ this contradicts the condition (C1). Otherwise, we have $|L(R\_Stem(T'))|= |L(R\_Stem(T))|,$ $|R\_Stem(T')| = |R\_Stem(T)|$ and $|L(T')|< |L(T)|$, this contradicts the condition (C2). \\
		Otherwise, we have $ x^+z_i \in E(G).$ We set $t\in N_{T}(x_i)\cap V(R\_Stem(T)).$ Since $G[x_i, t, x_{iz}, x]$ is not $K_{1,3}$-free we obtain either $xt\in E(G)$ or $x_{iz}t\in E(G)$ or $xx_{iz}\in E(G)$. Consider the spanning tree
		$$T':=\left\{\begin{array}{ll}T+\{xt,x^{+}z_i\}-\{xx^{+},x_ix_{iy}\},
		& \;\mbox{ if } xt\in E(G),\\
		T+x_{iz}t-x_ix_{iz},
		& \;\mbox{ if } x_{iz}t\in E(G),\\
		T+\{xx_{iz},x^{+}z_i\}-\{xx^{+},x_ix_{iz}\},
		& \;\mbox{ if } xx_{iz}\in E(G).\\
		\end{array}\right.$$
		Then we have $|L(R\_Stem(T'))|\leq |L(R\_Stem(T))|,$ $|R\_Stem(T')| \leq |R\_Stem(T)|,$ $|L(T')|\leq |L(T)|$ and $\sum_{i=1}^{l}\deg_{T'}(x_i)<\sum_{i=1}^{l}\deg_T(x_i).$ This violates the conditions (C0), (C1), (C2) or (C3).\\
		Subclaim \ref{claim2.8}.6 holds.

		On the other hand, it follows from Claim \ref{claim2.6} that.\\
		\noindent{\it Subclaim \ref{claim2.8}.7.} We have $N_{G}(w)\cap V(B_{y_i})=\emptyset$ for all $w\in L(R\_Stem(T))\setminus \{x_i\}.$
		
		By Subclaims \ref{claim2.8}.3-\ref{claim2.8}.4 we conclude that  $\{y_i\}, N_G(y_i) \cap V(P_T[y_i,x_i])$ and $ \left(N_G(\{z_i,x_i\}) \cap (P_T[y_i,x_i]\right)^{+} $ are pairwise disjoint subsets in $P_T[y_i,x_i]$. Combining with Claim \ref{claim2.5.1} and Subclaims \ref{claim2.8}.1, \ref{claim2.8}.6-\ref{claim2.8}.7, we obtain 
		\begin{eqnarray*} 
			\sum_{u\in U}|N_G(u) \cap V(B_{y_i}) |& =& | N_G(y_i) \cap V(B_{y_i})|+| N_G(z_i) \cap V(B_{y_i})|+| N_G(x_i) \cap V(B_{y_i})|\\
			&=& | N_G(y_i) \cap V(B_{y_i})|+| N_G(\{z_i,x_i\}) \cap V(B_{y_i})|\\
			&=&| N_G(y_i) \cap V(P_T[y_i,x_i])|+| (N_G(\{z_i, x_i\}) \cap P[y_i,x_i])^{+}|\\
			& \leq & |P_T[y_i,x_i]|- 1= |B_{y_i}|.
		\end{eqnarray*}
		This completes the proof of Claim~\ref{claim2.8}.
	\end{proof}
	Now, repeating the proof of Theorem \ref{thm1} for the subtree $R\_Stem(T)$ we obtain the following claim.
	\begin{claim}\label{claim-stem}
		$|N_{G}(L(R\_Stem(T)))\cap V(R\_Stem(T))|\leq |R\_Stem(T)|-|L(R\_Stem(T))|.$
	\end{claim}

	By Claim~\ref{claim2.4} and Claims~\ref{claim2.7}-\ref{claim-stem} we obtain that 
	\begin{eqnarray*}
		\deg_G(U)&=&\sum_{i=1}^{l} \bigg(\sum_{u\in U}|N_G(u)\cap V(B_{y_i}) |+\sum_{u\in U}|N_G(u) \cap V(B_{z_i}) |\bigg)+\\
		&+&\sum_{p\in L(T)\setminus U_1}\sum_{u\in U}|N_G(u) \cap V(B_{p}) |
		+ |N_{G}(L(R\_Stem(T)))\cap V(R\_Stem(T))|\\
		&\leq& \displaystyle \sum_{i=1}^{l}\vert B_{y_i}\vert + \displaystyle \sum_{i=1}^{l} \vert B_{z_i}\vert+\displaystyle \sum_{p\in L(T)\setminus U_1} \vert B_p \vert + \vert R\_Stem(T)\vert - |L(R\_Stem(T))|\\
		&=& \vert G \vert - \vert L(R\_Stem(T)) \vert\\
		&=& \vert G \vert - l.
	\end{eqnarray*}
	Hence 
	\begin{eqnarray*}
		\sigma_{3k+3}(G) &\leq& \sigma_{3l}(G)\leq \deg_G(U)\\
		&\leq & \vert G \vert -l \leq \vert G \vert -k-1.
	\end{eqnarray*}
	This contradicts the assumption of Theorem \ref{thm2}. Therefore, the proof of Theorem \ref{thm2} is completed.


\end{document}